\documentclass[a4paper,11pt]{article}

\usepackage{graphicx}
\usepackage{verbatim}
\usepackage{fancybox}
\usepackage{textcomp}
\usepackage{amsmath}
\usepackage{amsfonts}
\usepackage{amssymb}
\usepackage{color}  
\usepackage{amsthm}
\usepackage{mathrsfs}  
\usepackage{eurosym} 
\usepackage[utf8]{inputenc}
\usepackage[T1]{fontenc}
\usepackage{makeidx,bbm}
\usepackage[textwidth=100pt]{todonotes}
\usepackage{geometry}
 \usepackage[unicode=true,bookmarks=false,breaklinks=false,pdfborder={0 0 1},backref=false,colorlinks=false]{hyperref}

\newtheorem{theorem}{Theorem}[section]
\newtheorem{lemma}[theorem]{Lemma}

\newtheorem{corollary}[theorem]{Corollary}
\newtheorem{rem}[theorem]{Remark}
\newtheorem{fact}[theorem]{Fact}

\newcommand{\R}{\mathbb{R}}
\newcommand{\Z}{\mathbb{Z}}
\newcommand{\E}{\mathbb{E}}
\newcommand{\B}{\mathbb{B}}

\newcommand{\equnn}[1]{\begin{eqnarray*} #1 \end{eqnarray*}}
\newcommand{\equa}[1]{\begin{eqnarray} #1 \end{eqnarray}}
\newcommand{\ind}[1]{\mathbbm{1}_{\left[ {#1} \right] }}

\newcommand{\ee}{\mathbf e}
\newcommand{\bfa}{\mathbf a}
\newcommand{\bfq}{\mathbf q}

\newcommand{\tildeE}{\tilde{\mathbb E}}
\newcommand{\bgamma}{\boldsymbol\gamma}

\newcommand{\calA}{\mathcal A}
\newcommand{\calE}{\mathcal E}
\newcommand{\calH}{\mathcal H}
\newcommand{\dd}{d}
\newcommand{\Vare}{\mathbb V\!\mbox{ar}}
\newcommand{\Cove}{\mathbb C\mbox{ov}}

\title{A note on the discrete Gaussian Free Field \\ with disordered pinning on
  $\Z^d$, $d\geq2$}
\author{Loren Coquille, Piotr Mi\l{}o\'s}
\date{\today}

\begin{document}
\maketitle

\begin{abstract}
We study the discrete massless Gaussian Free Field on $\Z^d$, $d\geq2$,
in the presence of a disordered square-well potential supported on a finite strip
around zero. The disorder is introduced by reward/penalty
interaction coefficients, which are given by i.i.d. random variables. Under
minimal assumptions on the law of the environment, we prove that the quenched free energy associated to this model
exists in $\R^+$, is deterministic, and strictly smaller than the
annealed free energy whenever the latter is strictly
positive. 
\end{abstract}
\strut

  

\section{The model}  
We study the discrete Gaussian Free Field with a disordered
square-well potential. For $\Lambda$ a finite subset of $\Z^d$, denoted by $\Lambda\Subset\Z^d$, let $\varphi=(\varphi_x)_{x\in\Lambda}$
represent the heights over sites of
$\Lambda$. The values of $\varphi_x$ can also be seen as continuous
unbounded (spin) variables, we will refer to $\varphi$ as ``the
interface'' or ``the field''.
 
Let $\Omega=\R^{\Z^d}$ be the set of configurations. The finite volume Gibbs measure in $\Lambda$ for the discrete Gaussian Free Field with disordered square-well potential,
and $0$ boundary conditions, is the probability measure on $\Omega$
defined by :  
\equa{\label{astrip_measure}
\mu_{\Lambda}^{\ee}(d\varphi)=\frac1{Z_{\Lambda}^{\ee}}\exp\left({-\beta
    \calH_\Lambda(\varphi)+\beta\sum_{x\in\Lambda}
 (b\cdot e_x +h) \ind{\varphi_x\in[-a,a]}}\right)\prod_{x\in\Lambda}d\varphi_x\prod_{y\in\Lambda^c}\delta_0(d\varphi_y).}
where $a,\beta,b >0 $, $h\in\R$ and $\calH_\Lambda(\varphi)$ is given by 
\begin{equation}
	\calH_\Lambda(\varphi)=\frac1{4d}\sum_{\substack{\{x,y\}\cap\Lambda\neq\varnothing
	    \\ x\sim y}}(\varphi_x-\varphi_y)^2,\label{eq:gffHamiltonian}
\end{equation}
where $x\sim y$ denotes an edge of the graph $\Z^d$ and $\ind{A}$ denotes the indicator function of $A$. An environment is denoted as
$\ee:=(e_x)_{x\in\Lambda}$. We consider it to be an i.i.d\ family of random
variables such
that $\E(e_x)=0$, $\Vare(e_x)=1$, and 
\begin{equation}
\E(e^{b\cdot e_x+h})<\infty.\label{eq:assumption}
\end{equation}
The parameter $b$ is usually called the ``intensity of the disorder'', while $h$ is its average. We will see that \eqref{eq:assumption} is a minimal condition for the annealed free energy to be well-defined. The disordered potential attracts or repulses the field at heights belonging to $[-a,a]$. 
$Z_{\Lambda}^\ee$ is the partition function, i.e.\ it normalizes $\mu_{\Lambda}^\ee$ so it is a probability measure.
We stress that our model contains two levels of randomness. The first one is $\ee$ which we refer to as ``the environment''. The second one is the actual interface model whose low depends on the realization of $\ee$. 

The inverse temperature parameter $\beta$ enters only in a trivial way. Indeed, if we replace
the field $(\varphi_x)_{x\in\Lambda}$ by
$(\sqrt{\beta}\phi_x)_{x\in\Lambda}$, $a$ by $\sqrt{\beta}a$, and
$(b\cdot e_x+h)_{x\in\Lambda}$ by
$({\beta}(b\cdot e_x+h))_{x\in\Lambda}$ we have transformed the model to
temperature parameter $\beta=1$. In the sequel, we will therefore work
with $\beta=1$.

In this paper we focus on $d\geq2$ since 1-dimensional
models have been well-studied in the last decade
(see Section \ref{1d} below). The questions to address are the usual ones
concerning statistical mechanics models in random environment : Is the quenched
free energy non-random ? Does it differ from the annealed one ? Can we
give a physical meaning to the strict positivity (resp. vanishing) of
the free energy ? What can be said concerning the
quenched and annealed critical lines (surfaces) in the space of the relevant
parameters of the system ?  \\
We prove that the
quenched free energy exists, is non-random, and strictly smaller than the
annealed free energy whenever the annealed free energy is positive. 

In the forthcoming work \cite{CoqMil2013}  we also investigate the corresponding phase diagram : in the
  plane $(b,h)$, we can prove that the quenched critical line (separating the
phases of positive and zero free energy) lies strictly below the line
$h=0$. Thus there exists a non trivial region where the
field is localized though repulsed on average by the
environment. 

\section{Results}

We define the quenched (resp. annealed) free
energy per site in $\Lambda \Subset \Z^d$ by :
\begin{equation}
	f^\bfq_{\Lambda} (\ee)=|\Lambda|^{-1}\log \left(\frac{Z_{\Lambda}^{\ee}}{Z_\Lambda^0}\right) ,\quad
	f_{{\Lambda}}^\bfa(\ee)=|\Lambda|^{-1}\log\left(\frac{\E Z_{\Lambda}^{\ee}}{Z_\Lambda^0}\right), \label{eq:freeEnergyDef}
\end{equation}
where $Z_\Lambda^0$ denotes the partition function of the model with
no potential, $e_x \equiv 0$ (i.e. of the Gaussian free field). 
In the case when $\Lambda = \Lambda_n=\lbrace 0, ..., n-1\rbrace^d$
we will use short forms $f^\bfq_{n} (\ee)$ and $f^\bfa_{n} (\ee)$. By
the Jensen inequality, we have $f^\bfq(\ee)\leq
f^\bfa(\ee)$. Moreover, it
is not difficult to see that the annealed model corresponds to the
model with constant (we will also say homogenous) pinning with the strength
\begin{equation}\ell(\ee):=\log(\E(e^{b\cdot e_x+h})).\label{eq:effectiveAnnealedPotential}\end{equation}
for all $x\in\Lambda$. In other words $ \E Z^\ee_\Lambda=Z^{\ell(\ee)}_\Lambda$. We see now that assumption \eqref{eq:assumption} is  minimal for the annealed free energy to be well-defined.

 Below we present our results. The proofs are delegated to Section \ref{sec:proofs}.



\subsection{Existence and basic properties of the free energy}
In this section we present results concerning the existence and the positivity of the free energy. 

\begin{theorem}\label{thm:free_energy}  
Let $d\geq2$ and $\ee$ be an environment such that \eqref{eq:assumption} holds. Then the limit
$$ f^\bfq(\ee):=\lim_{\Lambda\uparrow\Z^d}f^\bfq_{\Lambda}(\ee)\in\R,$$
exists almost surely and in $L^2$, and does not depend on the sequence
$\Lambda\uparrow\Z^d$ provided that $|\partial\Lambda|/|\Lambda|\to0$. \\Moreover, $f^\bfq(\ee)$ is deterministic, i.e.\
$$ f^\bfq(\ee)=\E(f^\bfq(\ee)) \quad\mbox{a.s.}$$
\end{theorem}
The following is a straightforward corollary.
\begin{corollary}\label{thm:convergence_annealed}
Let $d\geq2$ and $\ee$ be an environment such that \eqref{eq:assumption} holds. Then the limit
$$ f^{\bfa} (\ee):=\lim_{\Lambda\uparrow\Z^d}f^\bfa_{\Lambda}(\ee),$$
exists in $\R$ and does not depend on the sequence
$\Lambda\uparrow\Z^d$ provided that $|\partial\Lambda|/|\Lambda|\to0$. 
\end{corollary}

\begin{fact} \label{lem:positivity_free_energy}
Let $d\geq2$ and $\ee$ be an environment such that \eqref{eq:assumption} holds. Then 
 $$f^{\bfa}(\ee)\geq f^{\bfq} (\ee)\geq 0.$$
 \end{fact}

{Moreover for the annealed model, using a variant of the argument in
the proof of \cite[Theorem 2.4]{BolVel2001}, it is possible to show that
\cite{Velenik:fk},}
\begin{fact}\label{fact:freeEnergyAnnealedEsitmates}
The annealed free energy is a non-decreasing function of
$\ell(\ee)$ such that $f^\bfa(\ee)=0$ whenever $\ell(\ee)<0$.
 Moreover, for $d\geq 3$ there exist a constant $C_d>0$ such that 
  $$
{  f^\bfa(\ee) = C_d \ell(\ee)
  (1+o(1))\quad \text{ as }\ell(\ee)\to0,
}
$$
For $d=2$ there exists a constant $C_2$ such that 
$$
{f^\bfa(\ee) = C_2 \frac{\ell(\ee)}{\sqrt{|\log(\ell(\ee))|}} (1+o(1)) \quad \text{ as }\ell(\ee)\to0
}
$$
\end{fact}

\subsection{Strict inequality between quenched and annealed free energies}
In this section we state the main results of our paper. Before that we recall that we assume $b>0$ in \eqref{astrip_measure} hence the disorder is always non-trivial. 

\begin{theorem}\label{thm:ineq_quenched_annealed}
Let $d\geq2$ and $\ee$ be an environment such that \eqref{eq:assumption} holds. Then 
\[
f^\bfq(\ee)<f^\bfa(\ee) \quad \mbox{whenever}\quad f^\bfa(\ee)>0.
\]
Moreover, the following quantitative bounds hold. Let
$\gamma:=\exp(b\cdot e_0+h-\ell(\ee))$, then
\begin{enumerate}
\item[(a)] For $d\geq3$ we have
\[
f^\bfq(\ee)-f^\bfa(\ee)\leq
\E \log(\lambda
  \gamma +1-\lambda),
\]
{where $\lambda:= \frac{C_1\ell(\ee)}{1+C_1 \ell(\ee)}$ for some $C_1 =C_1(a)>0$.}
\item[(b)] For $d=2$ we have 
$$
f^\bfq(\ee)-f^\bfa(\ee)\leq\E\log\left(\frac{\lambda}{|\log\lambda|}\gamma+1-\frac{\lambda}{|\log\lambda|}\right),$$
{for some $\lambda=\lambda(\ell(\ee))>0$ which equal to
$\frac{C_2\ell(\ee)}{\sqrt{|\log{\ell(\ee)}|}}$ with $C_2=C_2(a)>0$ for $\ell(\ee)$ small enough.}
\end{enumerate}
\end{theorem}

%
%
%

\begin{rem} 
\begin{enumerate}
\item The explicit expression for $\lambda(\ell(\ee))$ for large
  $\ell(\ee)$ in dimension 2 could be a
  priori derived by a method similar to the one developed in
  \cite{BolVel2001,IofVel2000}. One should keep track, though, of the dependency in $\ell(\ee)$ of
  the size of all the boxes. This information is of little relevance here.
\item{ It follows from Fact \ref{fact:freeEnergyAnnealedEsitmates} that
  in all dimensions $d\geq2$ we have
  $\lambda=\tilde C_d f^\bfa(\ee)$ for $\ell(\ee)$ small enough and
  some constants
  $\tilde C_d=\tilde C_d(a)>0$.}
\item
	We present two examples of the bounds for the concrete environment laws. Let the environment be given by the Bernoulli random variables, i.e. $\mathbb P(e_x = -1) = \mathbb P(e_x =1 )=1/2$. Then there exists a constant $C>0$ such that for $b,h$ small enough and $\ell(\ee)>0$ we have
	$$f^\bfq(\ee)-f^\bfa(\ee)\leq  -C b^2 (b^2/2+h) +o(b^2 (b^2/2+h)),$$
	for $d\geq 3$. We recall that in this case $f^\bfa(e) \approx \ell(\ee) \approx b^2/2+h$. Further for $d=2$ we obtain
	$$f^\bfq(\ee)-f^\bfa(\ee)\leq - C b^2 \frac{b^2/2+h}{|\log (b^2/2+h)|^{3/2}}+o\left(b^2 \frac{b^2/2+h}{|\log (b^2/2+h)|^{3/2}}\right).$$
	We note that the condition $\ell(\ee)>0$ yields $b^2/2+h >0$ hence the expression above is well-defined.\\
	The same estimates hold for the Gaussian environment i.e. $e_x
        \sim \mathcal{N}(0,1)$.
\end{enumerate}
\end{rem}

\section{Related results and open problems}
\subsection{Known results about homogenous models}\label{sec:intro_homogenous}
Our work relies on certain techniques developed in \cite{BolVel2001},\cite{DeuVel2000}. They focus on the case $e_x = const$ which we will refer to as the homogenous pinning model.
The review article, \cite{Vel2006}, focuses on the localization and delocalization of
random interfaces (in a non-random environment).\\
We briefly describe some results relevant to our work. Let us now consider the following model \footnote{Similar results
  hold for the homogenous version of our model with
  $b\cdot e_x+h=\varepsilon>0$ for all $x\in\Z^d$, see
  \cite{BolVel2001}.} :
\begin{eqnarray}\label{delta-pinning}
 \mu_\Lambda^{\varepsilon}(d\varphi)=\frac1{Z_\Lambda^{\varepsilon}}
\exp\left({-\calH_\Lambda(\varphi) }\right)
\prod_{x\in\Lambda}\left(d\varphi_x+\varepsilon\delta_0(d\varphi_x)\right)\prod_{y\in\Lambda^c}\delta_0(d\varphi_y),
\end{eqnarray}  
with $\varepsilon\geq 0$ and $\calH_\Lambda$ given by
\eqref{eq:gffHamiltonian}. 

We denote by $\mathcal{A}$ the (random) set of pinned sites i.e. $\mathcal{A}:=\{x\in\Lambda:\varphi_x =0\}$. Important results concerning the distribution of $\mathcal{A}$ were
obtained in \cite{BolVel2001}. First of all, it is strong FKG in the sense of
\cite{FKG1971}. Moreover, it can be compared with i.i.d.\ Bernoulli fields. Let $\B^\alpha_\Lambda$ be the
Bernoulli product measure with parameter $\alpha\in[0,1]$ in $\Lambda$. Namely the measure on subsets of $\Lambda$ given by 
$\B^\alpha_\Lambda(\mathcal{A}=A) = \alpha^{|A|}(1-\alpha)^{|\Lambda| - |A|}$. By \cite[Theorem 2.4]{BolVel2001} there exist constants $0<c_-(d)<c_+(d)<\infty$ such that
for any $\Lambda$, any $B\subset\Lambda$, and $\varepsilon$ sufficiently small we have 
\begin{eqnarray*}
\B^{c_-(d)g(\varepsilon)}_\Lambda(\calA\cap B=\varnothing)\leq
\mu^\varepsilon_\Lambda(\calA\cap B=\varnothing)\leq
\B^{c_+(d) g(\varepsilon)}_\Lambda(\calA\cap B=\varnothing),
\end{eqnarray*}
where
$$ g(\varepsilon)=\left\{\begin{tabular}{ll} 
$\varepsilon\vert\log\varepsilon\vert^{-1/2}$& $d=2,$\\
$\varepsilon$ & $d\geq3.$ \end{tabular}\right.$$
For $d\geq3$, an even stronger statement is true; the distribution of the pinned sites
 is strongly stochastically dominated by
$\B_\Lambda^{c_+(d)\varepsilon/(1+c_+(d)\varepsilon)}$ and strongly stochastically dominates
$\B_\Lambda^{c_-(d)\varepsilon/(1+c_-(d)\varepsilon)}$. Concerning the behavior of the interface, it is known that an
arbitrarily weak pinning $\varepsilon$ is sufficient to localize the
interface. Indeed, in \cite{DeuVel2000}, Deuschel and Velenik proved that, for a class of models including $\mu^\varepsilon_\Lambda$, 
 the infinite volume Gibbs measure, denoted by $\mu^{\varepsilon}$, exists in all
$d\geq1$. Further, for  $\varepsilon$ small enough and $K$ large
enough we have\footnote{ $a\asymp_d b$ means that there exist two constants $0 < c_1 \leq
    c_2 < \infty$, depending only on $d$, such that $c_1b
    \leq a \leq c_2b$.}
\begin{eqnarray*}
-\log\mu^{\varepsilon}(\varphi_0>K) \asymp_d 
\left\{\begin{tabular}{ll} 
$K$ & $d=1,$\\
$K^2/\log K$& $d=2,$\\
$K^2$ & $d\geq3.$ \end{tabular}\right.
\end{eqnarray*}
The so-called mass i.e. rate of exponential decay of the
two-point function associated to the infinite volume Gibbs
measure $\mu^{\varepsilon}$ is defined, for any $x\in\mathbb S^{d-1}$, by 
$$ m^\varepsilon(x):=-\lim_{k\to\infty}\frac1k\log \mu^\varepsilon (\varphi_0\varphi_{[kx]}).$$
where $[x]$ is the vector of integer parts of $x$'s
coordinates. In \cite{IofVel2000} Ioffe and Velenik showed that for any
$\varepsilon>0$ and $d\geq1$, 
\begin{eqnarray*}\label{positive_mass}
\inf_{x\in\mathbb S^{d-1}}m^\varepsilon (x)>0.
\end{eqnarray*}
The
localization of the interface becomes weaker as $\varepsilon\downarrow 0$. We
can quantify this by studying the behavior of the variance and the
mass of the field in this limit. The most precise results were obtained
by Bolthausen and Velenik in \cite{BolVel2001}. For $d=2$ and $\varepsilon$ small enough, 
\begin{eqnarray*}
  \mu^{\varepsilon}(\varphi_0^2)=\frac1\pi\vert\log\varepsilon\vert+O(\log\vert\log\varepsilon\vert).
\end{eqnarray*}
We recall that for $d\geq 3$ the variance is bounded even when $\varepsilon=0$. For $d\geq2$ and $\varepsilon$ small enough we have
\begin{eqnarray*} 
 m^\varepsilon \asymp_d 
\left\{\begin{tabular}{ll} 
$\sqrt{\varepsilon}\vert\log\varepsilon\vert^{-3/4}$ & $d=2$,\\
$\sqrt{\varepsilon}$& $d\geq3$.
\end{tabular}\right.\end{eqnarray*}

\subsection{Known results about disordered models}

\subsubsection{Models on $\Z$}\label{1d}

In \cite{AleSid2006}, Alexander and Sidoravicius studied the
1-dimensional model. They consider a
polymer, with monomer locations modeled by the trajectory of a Markov
chain $(X_i)_{i\in\Z}$, in the presence of a potential {(usually called a ``defect line'')} that interacts with the polymer
when it visits 0. Formally, let $V_i$ be an i.i.d.\ sequence of $0$-mean random variables, the model is given by weighting the realization of the chain with the Boltzmann term
$$ \exp\left(\beta\sum_{i=1}^n(u+V_i)\ind{X_i=0}\right).$$
They studied the localization transition in this model.
We say that the polymer is pinned, if a positive fraction of monomers is at 0. {In the plane $(\beta, u)$
critical lines are defined as follows: for $\beta$ fixed, let $u_c^\bfq(\beta)$
(resp. $u_c^{\bfa}(\beta)$) be the value of $u$ above which the polymer is pinned
with probability 1
(for the quenched (resp. annealed) measure)}. They showed that the quenched free
energy and critical point are non-random, calculated the critical point for a deterministic
interaction (i.e.\ $V_i \equiv 0$) and proved 
that the critical point in the quenched case is strictly smaller. 

When the underlying chain is a symmetric simple random walk
on $\Z$, the deterministic critical point is $0$, so having the quenched critical point $u_c(\beta)$ strictly
negative means that, even when the disorder is repulsive on average,
the chain is pinned. This result was obtained by Galluccio and Graber
in \cite{GalGra1996} for a periodic potential, which is frequently
used in the physics literature as a ``toy model'' for random
environments.

In \cite{GiaTon2005}, Giacomin and Toninelli investigated the order of
the localization transition in general models of
directed polymers pinned on a defect line. They proved that for quite a general class of disordered models the transition is at least of the second order,
i.e.\ the free energy is differentiable at the critical line
and the order parameter (which is the density of pinned sites) vanishes continuously at the transition.  This is particularly interesting as there are examples of ``non-disordered'' systems with the first order transition only (e.g. $(1+d)$-dimensional directed polymers for $d\geq5$, see \cite[Proposition 1.6]{Gia2007}). These results 
imply that the presence of a disorder may have a smoothening effect on the transition.

For 1-dimensional models, the
renewal structure of the return times to 0 plays important role, in particular it simplifies a lot of calculations. In \cite{Ale2008}, Alexander
{emphasized this fact} by assuming that the tails of the excursion length between consecutive returns of $X$ to $0$ are as $n^{-c}{\phi(n)}$ (for some $1 < c < 2$ and slowly varying $\phi$).


He analyzed the quenched and annealed critical
curves {in the plane $(u, \beta)$} for different values of $c$. He proved that for $c>3/2$ the quenched and annealed curves differ only when the temperature is high enough. For $c<3/2$ the quenched and annealed critical points are always equal. This confirmed a prediction made by theoretical physicists on the basis of the
so-called Harris criterion (see \cite[Section 5.5]{Gia2007}). The case $c=3/2$ had been open until recent papers \cite{GiaLacTon2010} and
\cite{GiaLacTon2011} of Giacomin, Lacoin and Toninelli. They proved that in this case the disorder is
relevant in the sense that the quenched critical point in shifted with
respect to the annealed one. They had considered i.i.d.\ Gaussian disorder
in the first paper and extended the result to more general i.i.d.\ laws,
as well as refined the lower bound on the shift,
in the second paper. Note that their results include pinning of a directed polymer in dimension (1+1) as already
mentioned, but also the classical models of two-dimensional wetting
of a rough substrate, pinning of directed polymers on a defect line in
dimension (3 + 1) and pinning of an heteropolymer by a point potential
in three-dimensional space.

\begin{rem} Paper \cite{Ton2008} contains a short
proof of the strict inequality between quenched and annealed free
energies in the case of the polymer model described above, with $V_i$
being Gaussian. The proof technique relies of the ``interpolation''
method well-known in the theory of spin glasses. Albeit we expect that
this approach can be extended to $d\geq 2$ its applicability is
limited to the Gaussian case. Any generalization beyond that will
probably be a non-trivial task. 
These topics will be addressed in a forthcoming paper.
\end{rem}

\subsubsection{Models on $\Z^d$, $d\geq2$}

The only result about random pinning models we are aware of is
\cite{JanDelVel2005}. In this paper, Janvresse, De La Rue and Velenik
considered the model (\ref{delta-pinning}) in dimension 1 and 2 with
$\varepsilon=\varepsilon_x\in\{0,\eta\}$, which models an interface
interacting with an attractive diluted potential. 
They show that the interface is localized {in a sufficiently
  large but finite box} (in the sense that there is a density of pinned sites) if and only if the sites at which the pinning potential is non-zero have positive density. Note that in this paper they characterize the set of realizations of
the environment for which pinning holds (the disorder is fixed, not sampled from some given distribution), which is  stronger than an almost sure result.
\\

We also mention a series of papers by Külske \emph{et al.}
(\cite{KulOrl2006}, \cite{EntKul2008}, \cite{KulOrl2008} which study a model with disordered magnetic field (instead of disordered pinning potential ).
For example, Külske and
Orlandi studied the following model in dimension 2
\begin{eqnarray}\label{model_random_field}
 \mu_\Lambda^{\varepsilon,(\eta)}(d\varphi)=\frac1{Z_\Lambda^{\varepsilon,(\eta)}}
\exp\left(- \frac{1}{4d}\sum_{x\sim y}V(\varphi_x-\varphi_y)
    +\sum_{x\in\Lambda}\eta_x\varphi_x\right)
\prod_{x\in\Lambda}\left(d\varphi_x+\varepsilon\delta_0(d\varphi_x) \right)
\prod_{y\in\Lambda^c}\delta_0(d\varphi_y),
\end{eqnarray}  
where  $(\eta_x)_{x\in\Lambda}$ is an arbitrary fixed
configuration of the external field and $V$ is not growing too slowly at infinity. Without disorder ($\eta\equiv0$), the interface is
localized for any $\varepsilon>0$
\cite{BolVel2001}. One could expect that in presence of disorder and at least
for very large $\varepsilon$ the interface is pinned. However, the
authors show that this is not the case : the interface diverges regardless of the pinning strength.
Thus the infinite-volume Gibbs measure for this model does
not exist. One could hope for the existence of the so-called
gradient Gibbs measure (Gibbs distributions of the increments
of the interface), which is a weaker notion. In
\cite{EntKul2008} Van Enter and Külske proved that such (infinite
volume) measures do not exist in dimension 2. 
  
\subsection{Open problems}
The model studied in this paper lends itself to number of extensions.
We list here a selection, with brief comments.
\begin{itemize}
\item \textbf{Path-wise~description~of~the~interface.} In the positive
free energy region, for our model, one expects localization, i.e.\ the finite variance
of $\varphi_{x}$ and exponential decay of correlations. A much more
difficult question concerns the behavior of the interface near the
critical line. Does it behave the
same as in the homogenous case (i.e.\ second order transition with
the density of pinned sites decreasing linearly for $d\geq3$, and
with a logarithmic correction for $d=2$) ? Or does the presence of
disorder have a smoothening effect on the transition (as it was proven
for certain 1-dimensional models) ?\\
In the zero free energy region, we expect the behavior similar to the entropic repulsion
for the GFF : in a box of size $n$ the interface should be repelled
at height $\pm\log n$ in $d=2$ and $\pm\sqrt{\log n}$ in $d\geq3$ (see \cite{DeuGia1999}
and references therein for details). The $\pm$ stems from the
fact that our model is symmetric with respect to reflection at zero height, hence
with probability $1/2$ it either goes upwards or downwards.

\item \textbf{Description~for~non~Gaussian~pair-potential.} A natural conjecture
is that the behavior of the model is the same if we change the
Gaussian term $(\varphi_{x}-\varphi_{y})^{2}$ to any other uniformly
convex potential $V(\varphi_{x}-\varphi_{y})$. Let us note however that the problem is difficult even in the homogenous case (due to lack of the Griffiths inequality).
%

\item \textbf{Non-nearest neighbors interactions.} We restricted our work to the case of the nearest
  neighbors interactions. {We suspect that the results holds true
for fast decaying interactions, at least with condition
  like in \cite[(2.1)]{BolVel2001} (which ensures a
control of the random walk's behavior in the random walk representations)}. As the behavior of the homogenous pinning model beyond this regime is not known, we are unable to pose any further conjectures. 

\item \textbf{Non i.i.d.\ environment~laws.} Going beyond the i.i.d.\ case is a
very interesting direction. Two natural cases would be the stationary
Bernoulli field or the quenched chessboard like configuration. These
questions may be closely connected to convexity/concavity properties
of the free energy function in the homogenous case. The
understanding of this case is still limited. It would be interesting
to know if finite range environment laws change the picture, as it is
sometimes the case in models with bulk disorder but it seems difficult
to answer to this question rigorously.

\item \textbf{Geometry~of~pinned~sites.} The geometry of
the pinned sites is still not fully understood in the homogenous
case. For the $d\geq3$ the law of pinned sites resembles a Bernoulli point
process. It is conjectured that once the
pinning tends to zero, under suitable
re-scaling, this field converges to Poisson point process. For $d=2$
the situation is not clear at all, since it is
  expected that the dependency
between the points will be preserved in the limit (implying the limit
being non-Poissonian). \\
Not only these questions propagate to the non-homogenous case but
also new ones arise. E.g. for our model it would
be interesting to study the joint geometry of
attractive and repulsive sites. 

\item \textbf{Models with wetting transition.} The effects of
  introducing a disorder in other models with pinning might be
  interesting, for example in models exhibiting a wetting
  phenomenon. In the case of the massless Gaussian model in $d=2$, it
  is known \cite{CapVel2000} that the wetting transition
takes place at a non-trivial point. A natural question to ask is,
if adding disorder shifts this point.

\end{itemize}

\section{Proofs}\label{sec:proofs}

\subsection{Proof of Theorem \ref{thm:free_energy}}
Although the proof is rather standard there are a few technical issues
which need to be clarified. They stem from the fact that we put only
minimal conditions of the environment. 

In the main proof it will be easier to deal
with environments with bounded support. Thus we start with a truncation
argument. Let $H>1$, we define a new environment
$\ee^H$ by $e^H_x := \tilde e_x\ind{\tilde e_x\in[-H,H]}$
with $\tilde e_x:=b e_x+h$. We claim that:
\begin{equation}
	\limsup_{H\uparrow\infty}\limsup_{\Lambda \uparrow
          \Z^d}|f^\bfq_{\Lambda} (\ee) - f^\bfq_{\Lambda} (\ee^H)|
        =0. \label{eq:freeComparison}
\end{equation}
\begin{proof}
We recall \eqref{eq:freeEnergyDef}. By a direct calculation we obtain
\begin{multline*}
	f^\bfq_{\Lambda} (\ee) - f^\bfq_{\Lambda} (\ee^H) =
        |\Lambda|^{-1} \log
        \mu_{\Lambda}^{\ee^H}\left(\exp\left(\sum_{x\in \Lambda} \tilde
            e_x\ind{\tilde e_x\in [-H,H]^c} \ind{\varphi_x \in [-a,a]}\right) \right) \\
	\leq |\Lambda|^{-1} \log \mu_{\Lambda}^{\ee^H}\left(\exp\left(\sum_{x\in \Lambda} \tilde e_x\ind{\tilde e_x\in [H,\infty)} \ind{\varphi_x \in [-a,a]}\right) \right) \leq |\Lambda|^{-1} \sum_{x\in \Lambda } \tilde e_x\ind{\tilde e_x\in [H,\infty)}.
\end{multline*}
As the size of the domain converges to infinity the last expression
converges to $\mathbb{E} \tilde e_{0} \ind{\tilde e_{0}\in [H,\infty)}$. This in turn, converges to $0$ as $H\to \infty$.
Further we observe that
\[
	f^\bfq_{\Lambda} (\ee) - f^\bfq_{\Lambda} (\ee^H) \geq
        |\Lambda|^{-1} \log
        \mu_{\Lambda}^{\ee^H}\left(\exp\left(\sum_{x\in \Lambda} \tilde
            e_x\ind{\tilde e_x\in [H,\infty)} \ind{\varphi_x \in
              [-a,a]}\right)\ind{ \varphi_x \notin [-a,a]\;\forall {x\in A}} \right),
\]
where $A = \{x\in \Lambda: \tilde e_x \leq -H\}$. Let further ${a_1, a_2, \ldots, a_N}$ be some enumeration of point in $A$ and $A_i = \{ a_1, a_2,\ldots, a_i \}$, where $i\in \{1,2,\ldots,N\}$. We write 
\begin{align}
	f^\bfq_{\Lambda} (\ee) - f^\bfq_{\Lambda} (\ee^H) &\geq
        |\Lambda|^{-1} \log \mu_{\Lambda}^{\ee^H}(\varphi_x\notin
        [-a,a]\;\forall{x\in A}) \nonumber\\
	&= |\Lambda|^{-1} \sum_{i=1}^N \log \mu_{\Lambda}^{\ee^H}(
        \varphi_{x_i}\notin [-a,a] | \varphi_x \notin [-a,a]\;\forall {x\in A_{i-1}} ). \label{eq:tmp19}
\end{align}
For any $x_i\in A$, by the spatial Markov property, we have
\begin{equation}\label{eq:gaussian_estimate}
	\mu_{\Lambda}^{\ee^H}(\varphi_{x_i}\notin [-a,a]|\varphi_x,
        x\sim x_i) = Z^{-1} \mathbf E\left(\ind{G \notin[-a,a]} e^{\tilde e_{x_i}\ind{G\in [-a,a]}}\right),
\end{equation}
where $G$ denotes a Gaussian random variable
$\mathcal{N}((\sum_{x\sim x_i } \varphi_{x})/2d, 1)$ {(which law is
denoted $\mathbf P$ and the corresponding expectation $\mathbf E$)} and $Z$ is the
normalizing constant $Z := \mathbf E\left( e^{\tilde e_{x_i}\ind{G\in [-a,a]}}\right)$. We recall that at site $x_i$ the we have
necessarily $\tilde e_{x_i}<0$ hence $Z\leq 1$. Finally, the
expression \eqref{eq:gaussian_estimate} is lower bounded by
\[
	\mathbf E\left(\ind{G \notin[-a,a]} e^{\tilde e_{x_i}\ind{G\in [-a,a]}}\right) = \mathbf{P}(G \notin[-a,a])\geq C,
\]
for certain $C>0$ independent of $(\sum_{x\sim x_i } \varphi_{x})/2d$. We use this to estimate \eqref{eq:tmp19} as follows
%
%
\[
	f^\bfq_{\Lambda} (\ee) - f^\bfq_{\Lambda} (\ee^H) \geq 
        |\Lambda|^{-1} N \log C\rightarrow \mathbb{P}(\tilde e_x
        \leq -H) \cdot \log C\quad\mbox{as}\quad\Lambda\uparrow\Z^d.
\]
Alike previously this converges to $0$ as $H\to \infty$, which
concludes the proof of \eqref{eq:freeComparison}.
\end{proof} 
Further we will
assume that the law of the environment has a bounded
support. Moreover, by the well-known fact that the free energy of the Gaussian
free field exists, i.e. $\lim_{\Lambda \uparrow \Z^d} |\Lambda|^{-1}\log
Z_\Lambda^0$ exist and is finite, it is enough to prove the existence
of the limit of $\tilde f^\bfq_{\Lambda} (\ee)$ given by
\[
 \tilde f^\bfq_{\Lambda} (\ee) := |\Lambda|^{-1} \log Z_\Lambda^\ee.
\]

We prove convergence only along a sequence of boxes
$B_n=\Lambda_{2^n-1}$. The extension to the case of general sequences
is standard.
We will cut $B_n$ in $2^d$ sub-boxes denoted by $B_{n-1}^{i}$. Let $X=(\bigcup_{i=1}^{2^d}\partial
B_{n-1}^{i})\backslash\partial B_n$ be ``the border'' between the sub-boxes. 
In order to prove the
existence of the limit along $(B_n)_n$, we first derive a ``decoupling
property''.
Namely, there exists $c_n \geq 0$ such that $\sum_n
c_n<\infty$ and $$|Z^{\ee}_{B_n}-\prod_{i=1}^{2^d}Z^{\ee^{i}}_{B^{i}_{n-1}}|\leq c_n,$$
for any realization of $\ee$, where $\ee^{i}$ is the restriction
  of $\ee$ to the box $B_{n-1}^{i}$. 

The next lemma provides us with control over the concentration of the field. 
\begin{lemma}\label{lem:gaussian_tail_model} 
	There exist $C_1,C_2, C_3>0$ such that for $n$
        sufficiently large, for all $x\in \Lambda_n$ and $T>0$  we have
	 $$\mu ^{\ee}_{\Lambda_n}(|\varphi_x|>T+C_3 \log n)\leq C_1e^{-C_2 T ^2/\log n}.$$
\end{lemma}    

\begin{proof} Let $\calA:=\{x : \varphi_x\in[-a,a]\}$. We first use
  the following decomposition, {which is a reordering of the
    pinning contributions over subsets of $\Lambda$}:
\equnn{
\mu ^{\ee}_{\Lambda_n}(|\varphi_x|>T)&=&\sum_{A\subset\Lambda}
{\left(\left(\prod_{x\in A}e^{be_x+h}\right)\frac{Z^0_{\Lambda_n}(\calA=A)}{Z^\ee_{\Lambda_n}}\right)}
\mu ^{0}_{\Lambda_n} (|\varphi_x|>T \;\vert\;\calA=A).
}
It is sufficient to upper-bound the rightmost term
uniformly in $A$.
Using the FKG inequality (see e.g. \cite[Section B.1]{Gia2001}) it is standard to check that 
\[
	\mu ^{0}_{\Lambda_n} (d\varphi_x\;\vert\;\calA=A)
	\prec \mu ^{0}_{\Lambda_n}(d\varphi_x \,|\, \forall y\in \Lambda_n,\;
	\varphi_y \geq a),
\]
where $\prec$ denotes the stochastic domination. Further, we intend to use the Brascamp-Lieb inequality. To this end we first estimate
\[
	A_n:=\mu ^{0}_{\Lambda_n}(\varphi_x \,|\, \forall y\in \Lambda_n ,\;\varphi_y \geq a ).
\]
For $d\geq 3$ it follows easily by \cite[Theorem 3.1]{Gia2001} that $A_n \leq C \sqrt{\log n}$ for some $C>0$. For $d=2$ let us assume first that $A_n \geq C \log n$ for some large $C$. By \cite[(B.14)]{Gia2001} we have $\mu ^{0}_{\Lambda_n}(\varphi_i^2\,|\, \forall y\in \Lambda_n ,\;\varphi_y \geq a ) \leq \log n + A_n^2$. Using the Paley-Zygmund inequality we get 
\[
	\mu ^{0}_{\Lambda_n}(\varphi_i \geq A_n/2 \,|\, \forall y\in \Lambda_n ,\;\varphi_y \geq a ) \geq \frac14\cdot \frac{A_n^2}{\log n + A_n^2} \geq 1/8,
\]
for $n$ large enough. Increasing $C$ further, if necessary, we get a contradiction with \cite[Theorem 4]{BolDeuGia2001}. Thus, $A_n \leq C \log n$. Finally by the Brascamp–Lieb inequality \cite[Section B.2]{Gia2001} we obtain
\begin{eqnarray*}
	\mu ^{0}_{\Lambda_n}(\varphi_x \geq T \,|\, \forall y\in
        \Lambda_n ,\;\varphi_y \geq a)
        \,\leq\, \exp( - \tilde C_2
        (T-A_n)^2/\log n)\,\leq\, C_1\exp(-C_2T^2/\log n)
\end{eqnarray*}
\end{proof}
Let us define $\tilde X=X\cup(\partial X\cap B_n)$, which is a thickening
of $X$ consisting of 3 ``layers''. Then, Lemma~\ref{lem:gaussian_tail_model}
allows us to control the height of the field on $\tilde X\subset B_n$. Let us fix some $\delta>0$, we have
\begin{equation}
	\label{height_on_X}
	\mu^{\ee}_{B_n}(\exists {i\in \tilde X} :|\varphi_i|>2^{\delta n})\leq
	\sum_{i\in \tilde X}\mu^{\ee}_{B_n}(|\varphi_i|>2^{\delta n}) \leq
	e^{-C 2^{\delta n}},
\end{equation}
for some $C>0$. We  denote the full Hamiltonian of
our system (i.e. not only the Gaussian part) by $\mathcal{K}(\varphi)$. Using \eqref{height_on_X} we deduce
\[
	Z^{\ee}_{B_n}\leq (1+2e^{-C 2^{\delta n}})\int_{\R^{B_n}} \ind{|\varphi_x|\leq 2^{\delta
	    n} , \forall {x\in \tilde X}} e^{-\mathcal{K}(\varphi)} d\varphi.
\]
Given $\varphi$ we define $\tilde{\varphi}$ by setting $\varphi_x =0$ for $x\in X$ and $\tilde{\varphi}_x = \varphi_x$ otherwise. It is easy to check that for $\varphi$ fulfilling $|\varphi_x|\leq 2^{\delta
    n}$  for any $x\in \tilde X$ we have $|\mathcal{K}(\varphi) - \mathcal{K}(\tilde{\varphi})|\leq C|X| 2^{2\delta n} $, for some $C>0$ (here we used the fact that the support of the law of the environment is bounded). Therefore

\[
	Z^{\ee}_{B_n}\leq (1+C'e^{-C 2^{\delta n}}) e^{C|X|2^{2\delta n}}\int_{[-2^{\delta n},2^{\delta n}]^X} \prod_{i} \int_{\R^{B^{i}_n}} \ind{|\varphi_x|\leq 2^{\delta
	    n} , \forall {x\in \tilde X}} e^{-\mathcal{K}(\tilde{\varphi})} d\varphi.
\]
We notice that $\tilde \varphi$ ``enforces'' $0$ boundary conditions on $X$. Consequently, each of the inner integrals is bounded from above by $Z^{\ee^{i}}_{B_n^{i}}$. This leads to \
\[
	\tilde f^\bfq_{B_n}(\ee)\leq
	\frac1{2^d}\sum_{i=1}^{2^d} \tilde f^{\bfq}_{B^i_{n-1}}(\ee^{i}){+C_{max} 2^{n(2\delta-1)},}	
\]
for some $C_{max}\geq 0$.
Now let us prove a bound from below. We have
\[
	Z^{\ee}_{B_n}=\int_{\R^{B_n}} e^{-\mathcal{K}(\varphi)}\dd{\varphi}\geq
	\int_{\R^{B_n}} \ind{\varphi \in A} e^{-\mathcal{K}(\varphi)} \dd{\varphi},
\]
{where $A = \{\varphi: \varphi_x \in (-a,a),\forall{x\in X} \text{ and
} |\varphi_x|\leq n^3,\forall {x\in \tilde X}\}$.} Let $\varphi \in
A$ we define $\tilde{\varphi}$ by setting $\varphi_x =0$ for $x\in X$
and $\tilde{\varphi}_x = \varphi_x$ otherwise. It is easy to check
that $|\mathcal{K}(\varphi) - \mathcal{K}(\tilde{\varphi})|\leq C|X|
n^3 $, for some $C>0$ (again we use the fact that that the environment
is bounded). Therefore
\[
	Z^{\ee}_{B_n} \geq e^{-C|X| n^3} \int_{[-a,a]^X} \prod_{i}  \int_{\R^{B_n^{i}}} \ind{\varphi \in A} e^{-\mathcal{K}(\tilde{\varphi})}\dd{\varphi}.
\]
Let us consider an integral in the product. Using Lemma
\ref{lem:gaussian_tail_model} similarly as before we check that it is
bounded from below by $ \left(1 -|X| e^{-Cn}\right)
Z^{\ee^{i}}_{B_n^{i}}$. Hence
\[
	\tilde f^\bfq_{B_n}(\ee)\geq
	\frac1{2^d}\sum_{i=1}^{2^d} \tilde f^{\bfq}_{B_{n-1}^i}(\ee^{i})-C_{\min}2^{-n/2}.
\]
for some constant $C_{\min}\geq 0$. Combining our two bounds above we obtain
\equa{\label{bounds_free_energy}
C_{\min}2^{-n/2}\leq \tilde
f^\bfq_{B_n}(\ee)-\frac1{2^d}\sum_{i=1}^{2^d} \tilde f^{\bfq}_{B^i_{n-1}}(\ee^{i})\leq C_{\max}{2^{n(2\delta-1)}}.
}
By (\ref{bounds_free_energy}), it is easy to see that :
\begin{equation}
	\label{bound_expectation}
	|\E(\tilde f^\bfq_{B_n}(\ee)) -\E(
	\tilde f^{\bfq}_{B_{n-1}}(\ee))|\leq |\E(\tilde f^\bfq_{B_n}(\ee)) -\E\Big(\frac1{2^d}\sum_{i=1}^{2^d}
	\tilde f^{\bfq}_{B_{n-1}^i}(\ee^{i})\Big)|\leq
	{ 2^{ - c n} },
\end{equation}
for some $c>0$ as soon as $\delta<1/2$ and $n$ is large enough. The right hand side is thus summable hence
$\E(\tilde f^\bfq_{B_n}(\ee))$ converges as $n\rightarrow\infty$ to some limit $L\in \R$.
Using independence of environment among the boxes $B_{n-1}^{i}$,
 we can estimate the variance. Let us write
$$ \tilde f_{B_n}^{\bfq}(\ee)=2^{-d}\sum_{i=1}^{2^d} \tilde f _{B_{n-1}^i}^{\bfq}(\ee^{i})+\calE_n, $$
where $\calE_n$ is the (random) error term above. We have $|\calE_n|\leq 2^{-cn}$. Thus
\begin{eqnarray*}\Vare(\tilde
  f_{B_n}^{\bfq}(\ee))&=&\Vare\Big(2^{-d}\sum_{i=1}^{2^d} \tilde f
  _{B_{n-1}^i}^{\bfq}(\ee^{i})\Big)+\Vare(\calE_n)+
  \Cove\Big(2^{-d}\sum_{i=1}^{2^d}\tilde f _{B_{n-1}^i}^{\bfq}(\ee^{i}),\calE_n\Big)\\
&\leq &2^{-d}\Vare(\tilde f _{B_{n-1}}^{\bfq}(\ee))+\Vare(\calE_n)+
2^{-cn}\E \Big( 2^{-d}\sum_{i=1}^{2^d}\tilde f _{B_{n-1}^i}^{\bfq}(\ee^{i})\Big). 	 
\end{eqnarray*}
Now, since $(\E(\tilde f _{B_{n-1}}^{\bfq}(\ee)))_n$ converges, we get the following inductive upper bound on the
variance:
\[
	\Vare(\tilde f_{B_n}^{\bfq}(\ee)) \leq 2^{-d}\Vare(\tilde f _{B_{n-1}}^{\bfq}(\ee)){+2^{-2cn}+ 2L 2^{-cn}},
\]
as long as $n$ is large enough.  We deduce that for some $b\in(0,1)$
and $n$ large enough we have $\Vare(\tilde f_{B_n}^{\bfq}(\ee)) \leq
b^n$. This yields $\tilde f^\bfq_{B_n}(\ee) \rightarrow L$ both in $L^2$ and almost surely.

\subsection{Proof of Fact \ref{lem:positivity_free_energy}}

The first inequality follows by the Jensen inequality. We recall definition \eqref{eq:freeEnergyDef}. We have
$$
f^{\bfq} (\ee)=\lim_{n\to\infty}n^{-d}\log\mu_{\Lambda_n}
^{0}\left(\exp\left(\sum_{x\in\Lambda_n}(b\cdot e_x+h)\ind{|\varphi_x|\leq
      a}\right)\right) \geq \lim_{n\to\infty}n^{-d}\log\mu_{\Lambda_n}
^{0}\left(\varphi_x \geq a\,,\, \forall x\in \Lambda_n \right).
$$
One easily checks that
\begin{multline*}
 \mu_{\Lambda_n}
^{0}(\varphi_x \geq a\,,\, \forall x\in \Lambda_n )\\ \geq \mu_{\Lambda_n}
^{0}(\varphi_x \geq a\,,\, \forall x\in \Lambda_{n-1} |
  \varphi_{x}\in[a,a+1]\,,\, \forall x\in \partial \Lambda_{n-1})\cdot
  \mu_{\Lambda_n}
^{0}(\varphi_x \in [a,a+1] \,,\, \forall x\in \partial \Lambda_{n-1} ).
\end{multline*}
The second factor is of order $e^{-C n^{d-1}}$
and thus is negligible. The first one is bounded from below by
$I_n=\mu_{\Lambda_{n-1}}^0(\varphi_x \geq 0\,,\,\forall x \in
\Lambda_{n-1})$. Let us fix some $\delta \in (0,1)$. By the FKG
inequality we get 
\begin{align*}
  n^{-d} \log I_n &\geq n^{-d} \log \mu_{\Lambda_{n-1}}^0(\varphi_x \geq 0\,,\,\forall x \in
  \Lambda_{\lceil \delta n \rceil}) + n^{-d} \sum_{x \in \Lambda_{n-1} \setminus \Lambda_{\lceil\delta
  n \rceil}}  \log \mu_{\Lambda_{n-1}}^0(\varphi_x \geq 0) \\
&\geq (-C)(n^{-1} + (1-\delta)^d),\quad\mbox{for some constant }C>0.
\end{align*}
The estimation of the second term follows simply
by the fact that $\mu_{\Lambda_{n-1}}^0(\varphi_x \geq 0)=1/2$,
whereas to
treat the first term we use the entropic repulsion results: \\
writing
$\Omega^+_{\Lambda_{\lceil\delta n\rceil}}=\{\varphi_x\geq0\;,\forall
x\in\Lambda_{\lceil\delta n\rceil}\}$, in dimension 2 we know that
$\mu_{\Lambda_n}^0(\Omega^+_{\Lambda_{\lceil\delta n\rceil}})=\exp(-O((\log
n)^2))$ (see \cite[Theorem 3]{BolDeuGia2001}) while in dimension 3 and higher
the above probability is of order $\exp(-O(n^{d-2}\log n))$ (see \cite[Theorem 1.1]{BolDeuZei1995}). Taking
limits $n\to\infty$ and $\delta \to 1$ we obtain
$\liminf_{n\rightarrow \infty }n^{-d}\log I_n \geq 0$, which
concludes the proof.
\qed

\subsection{Proof of Theorem \ref{thm:ineq_quenched_annealed}}\label{sec:quenchedVsAnnealed}
We recall \eqref{astrip_measure} and we  perform calculations which resemble the so-called high-temperature expansion
\begin{eqnarray}\label{pinned_sites_repr_model}
\mu_{\Lambda}^{\ee}(d\varphi)&=&\frac1{Z_{\Lambda}^{\ee}}\exp\left({-
    \calH_\Lambda(\varphi)+\sum_{x\in\Lambda}
 (b\cdot e_x+h)
  \ind{\varphi_x\in[-a,a]}}\right)\prod_{x\in\Lambda}d\varphi_x\prod_{y\in\Lambda^c}\delta_0(d\varphi_y)\\
&=&\frac1{Z_{\Lambda}^{\ee}}\exp\left(-
    \calH_\Lambda(\varphi)\right) \prod_{x\in\Lambda}\left((e^{
      b\cdot e_x+h}-1)\ind{\varphi_x\in[-a,a]}+1\right)\prod_{x\in\Lambda}d\varphi_x\prod_{y\in\Lambda^c}\delta_0(d\varphi_y)\nonumber\\
&=&\sum_{A\subset\Lambda}
\underbrace{\left(\prod_{x\in A} \left(e^{
      b\cdot e_x+h}-1\right)\frac{Z ^{0}_\Lambda(\varphi_x\in[-a,a]
  \;,\forall x\in A)}{Z ^{\ee}_{\Lambda}}\right)} _{=\nu_{\Lambda}^{\ee} (A)}
\mu ^{0}_\Lambda(d\varphi \;\vert\; \varphi_x\in[-a,a]
\;,\forall x\in A).\nonumber
\end{eqnarray}
We observe that when $b\cdot e_x+h\geq 0$ for any $x$ in the domain $\Lambda$, then $\nu_\Lambda^\ee$
is a probability measure (otherwise some weights $e^{b\cdot e_x+h}-1$
are negative).\\
 For any homogenous environment
$e_x=\varepsilon$ for all $x\in\Lambda$, it is known \cite{BolVel2001} that $\nu^\varepsilon_{\Lambda}$ is strong FKG in the sense of
\cite{FKG1971}, and that it can be stochastically majored
and minored by two Bernoulli product
measures (the precise statement needed in our work will appear later on).

By Theorem \ref{thm:free_energy}, Corollary \ref{thm:convergence_annealed} and \eqref{eq:effectiveAnnealedPotential} we have
\begin{align}\label{eq:diff_free_energies}
f^\bfq(\ee)-f^\bfa(\ee)&=
\lim_{n\to\infty} n^{-d}
\E\log\left(\frac{Z_{\Lambda_n}^{\ee}}{Z_{\Lambda_n}^{\ell(\ee)}}\right)\nonumber\\
&=\lim_{n\to\infty}
n^{-d} \E\log\mu_{\Lambda_n}^{\ell(\ee)}\left(\exp\left(\sum_{x \in \Lambda_n}(b\cdot
  e_x+h-\ell(\ee))\ind{\varphi_x\in[-a,a]}\right)\right).
\end{align}
For the rest of the proof $\mathcal{A}$ will denote the set of
``pinned points'' of a given
configuration $\varphi \in \Omega$, namely
$\mathcal{A} = \{x\in \Lambda_n: \varphi_x \in [-a,a] \}.$
 By \eqref{eq:diff_free_energies} we conclude that our goal is to prove
$$\limsup_{n\to\infty}n^{-d}\E\log\left(\mu_n^{\ell(\ee)}\left(\bgamma^\mathcal
	    A\right)\right)<0,$$ where $\mu^{\ell(\ee)}_n$ is a
        simplified notation for
        $\mu^{\ell(\ee)}_{\Lambda_n}$ and we denote
       $$\bgamma^A:=\prod_{x\in
          A}\gamma_x\quad\text{with}\quad\gamma_x:=\exp(b\cdot e_x+h-\ell(\ee)).$$

Let us now comment on the proof strategy. Let us observe that the
calculations would be simple if $\mathcal{A}$ was distributed
according to an i.i.d.\ Bernoulli($\lambda$) product measure. 
Indeed, in such a case, the above limit does not depend on n:
\begin{equation}
n^{-d}\E\log\mu_n^{\ell(\ee)}(\bgamma^\calA)=n^{-d}\E \log\left(\prod_{x\in \Lambda_n} (\lambda\gamma_x +1 -\lambda) \right) = \E \log(\lambda\gamma_0 +1 -\lambda )<0.\label{eq:bernoulliCalculations}
\end{equation}
The last inequality follows by the strict concavity of the logarithm and the Jensen inequality. However, the
interaction between the geometry of $\mathcal{A}$
and the one of the environment in $\gamma^\mathcal A$ might be potentially complicated
and hard to analyze. Exploiting the fact that the environment is
i.i.d.\ we will introduce an additional randomization. This will simplify
the problem so that only the information about the cardinality of
$\mathcal{A}$ will matter. The last trick we use is to compare the distribution of $\mathcal{A}$ under
$\mu_n^{\ell(\ee)}$ with the measure $\nu_n^{\ell(\ee)}$ defined in
\eqref{pinned_sites_repr_model}. This will enable to use the
stochastic domination results announced above and to get an expression similar to \eqref{eq:bernoulliCalculations}. Further, calculations are standard though little tiresome
since we work with general laws of environments.
                
Let us now introduce the randomization. Let $\pi$ be a permutation of
the vertices of $\Lambda_n$ chosen uniformly at random. We will
denote the corresponding expectation by $\tildeE$. It is easy to check that for any i.i.d.\ pinning law $\E\tildeE(\cdot)=\E(\cdot)$. By the Jensen inequality we have
\begin{eqnarray} \label{eq:tmp16}
	 n^{-d}\E\log\left(\mu_n^{\ell(\ee)}\left(\bgamma^\mathcal
	    A\right)\right)&=&
	n^{-d}\E\tildeE\log \mu_n^{\ell(\ee)}\left(\prod_{i\in\mathcal{A}}\gamma_{\pi(i)}\right)\\
	&\leq& n^{-d}\E\log\tildeE
        \mu_n^{\ell(\ee)}\left(\prod_{i\in\mathcal{A}} \gamma_{\pi(i)}\right)=
	n^{-d}\E\log\tildeE \mu_n^{\ell(\ee)}\left(\bgamma^{\pi(\mathcal{A})}\right),\nonumber
\end{eqnarray}
where $\pi(A)=\lbrace\pi(i):i\in A\rbrace$. Intuitively,  $\tildeE \mu_n^{\ell(\ee)}$ is the expectation of the distribution of pinned
sites ``scattered'' by a random permutation. Thanks to this we can work
with a uniformly distributed set of pinned points, provided we know
its cardinality. More precisely,
\begin{equation}
\tildeE \mu_n^{\ell(\ee)}\left(\bgamma^{\pi(\mathcal A)}\right)=\sum_{k=0}^{n^{d}}
\binom{n^d}{k}^{-1}\left(\sum_{A\subset\Lambda_n : \vert A\vert=k} \bgamma^A\right)
\mu_n^{\ell(\ee)}\left(\vert \mathcal{A}\vert=k\right). \label{eq:basicEstimate}
\end{equation}
We recall measure $\nu_n^{\ell(\ee)}$ defined in \eqref{pinned_sites_repr_model}. Paper \cite{BolVel2001} provides us with stochastic domination results which will be useful in our estimations. To this end we make the following 
elementary calculations:
\begin{eqnarray}\label{comparison_nu_tildenu}
\mu_{n}^{\ell(\ee)}(|\calA|\leq
k)
&=&\sum_{A\subset\Lambda_n}\nu_{n}^{\ell(\ee)}(A)\cdot \mu^{0}_{n}(|\mathcal{A}|\leq
k \,|\, \forall x\in A, |\varphi_{x}|\leq a)\nonumber\\
&=&\sum_{A\subset\Lambda_n,|A|\leq
  k}\nu_{n}^{\ell(\ee)}(A)\cdot\mu^{0}_{n}(|\mathcal{A}| \leq k \, |\,\forall x\in A,|\varphi_{x}|\leq a) \nonumber\\
&\leq&\sum_{A\subset\Lambda_n,|A|
\leq
  k}\nu_{n}^{\ell(\ee)}(A) =\nu_{n}^{\ell(\ee)}(|\mathcal{A}|\leq k).
\end{eqnarray}  
The next difficulty is that $\{|\mathcal{A}|=k\}$ appearing in
\eqref{eq:basicEstimate} is not an increasing event. This will be
handled differently for $d\geq 3$ and $d=2$. Let us start with the
former. 

\paragraph{Case $d\geq3$.}
By \cite[Theorem 2.4, (2.15)]{BolVel2001}, there exists some $C_1>0$
such that $\nu_n^{\ell(\ee)}$ stochastically dominates a Bernoulli
product measure, denoted 
$\mathbb B_n^\lambda$, with a specific intensity $\lambda$ depending on $a$ and
$\ell(\ee)$. More precisely,
\begin{equation}\label{stoch_dom_thm}
 \nu_n^{\ell(\ee)} \succ \mathbb B_n^\lambda \quad \mbox{ with
 }\quad \lambda:=C_1\ell(\ee)/(1+C_1 \ell(\ee)).
\end{equation}
Below we will write $\mathbb B_n^\lambda(|\calA|=k)=b_{n,\lambda}(k)=\binom{n^{d}}{k}\lambda^{k}(1-\lambda)^{n^d-k}$. Observe that by
\cite{BolVel2001} we know that
$f^\bfa(\ee)>0$ as soon as $\ell(\ee)>0$ and consequently $\lambda>0$. As
the event 
$\lbrace\vert \calA\vert\leq k\rbrace$ is decreasing, we have the
following upper-bound, using \eqref{comparison_nu_tildenu} and \eqref{stoch_dom_thm} :
\begin{eqnarray}\label{stoch_dom}
\mu_n^{\ell(\ee)}\left(\vert
  \calA\vert=k\right)
&\leq& \mu_n^{\ell(\ee)}\left(\vert
  \calA\vert\leq k\right)
\leq \nu_n^{\ell(\ee)}\left(\vert
  \calA\vert\leq k\right)\leq \mathbb B_n^\lambda\left(\vert
  \calA\vert\leq k\right) \\
&=& \sum_{j=0}^k b_{n,\lambda}(j)=
b_{n,\lambda}(k) {\left(1+\sum_{j=0}^{k-1}\frac{b_{n,\lambda}(j)}{b_{n,\lambda}(k)}\right) }=
b_{n,\lambda}(k)\left(1+\sum_{j=0}^{k-1}\prod_{i=j}^{k-1}\frac{b_{n,\lambda}(i)}{b_{n,\lambda}(i+1)}\right),\nonumber
\end{eqnarray}
Now, for $i\leq\lfloor\lambda n^d\rfloor$,
$$\frac{b_{n,\lambda}(i)}{b_{n,\lambda}(i+1)}=\frac{i+1}{n^d-i}{\frac{1-\lambda}{\lambda}}\leq1.$$
This gives an upper bound for $k\leq \lfloor\lambda n^d\rfloor$, namely
\begin{equation}
	\mu_n^{\ell(\ee)}\left(\vert
	  \calA\vert=k\right)\leq n^d\cdot\mathbb B_n^\lambda\left(\vert
	  \calA\vert=k\right). \label{eq:intermediateEstimate}
\end{equation}    
 By Stirling's formula for  $k\in [\lceil\lambda n^d\rceil, n^d]$ we have $ b_{n,\lambda_k}(k)\geq c_2 n^{-d/2}>0$, where $\lambda_k:= k n^{-d}$. Trivially
$$\mu_n^{\ell(\ee)}\left(\vert
  \calA\vert=k\right)\leq n^{d}\cdot\mathbb B^{\lambda_k}_n\left(\vert
  \calA\vert=k\right),$$
for $n$ large enough. Using the above estimates we treat \eqref{eq:basicEstimate} as follows
\begin{eqnarray*}\tildeE \mu_n^{\ell(\ee)}\left(\bgamma^{\pi(\mathcal
      A)}\right)
&\leq&n^d\sum_{k=0}^{n^d}\binom{n^d}{k}^{-1}\left(\sum_{A\subset\Lambda_n :
  |A|=k}\bgamma^A \right)\left[
\mathbb B_n^\lambda\left(\vert
  \calA\vert=k\right)+
\sum_{j=\lceil\lambda n^d\rceil}^{n^d}\mathbb B^{\lambda_j}_n\left(\vert
  \calA\vert=k\right) \right]\\
&\leq&n^{2d}\max_{\alpha\in[\lambda,1]}\sum_{k=0}^{n^d}\binom{n^d}{k}^{-1}\left(\sum_{A\subset\Lambda_n :
  |A|=k}\bgamma^A\right)\cdot\mathbb B^{\alpha}_n (|\calA|=k)\\
&\leq&  n^{2d}\max_{\alpha\in[\lambda,1]}
\mathbb B^{\alpha}_n\left(\bgamma^\mathcal A\right) = n^{2d}\max_{\alpha\in[\lambda,1]}\prod_{x\in\Lambda_n}\mathbb B^{\alpha}_n\left(\gamma_x^{\ind{x\in\calA}}\right).
\end{eqnarray*}
Hence, 
$$ n^{-d}\log\tildeE \mu_n^{\ell(\ee)}\left(\bgamma^{\pi(\mathcal
    A)}\right)\leq
\left(n^{-d}\max_{\alpha\in[\lambda,1]}\sum_{x\in\Lambda_n}\log(\alpha \gamma_x+(1-\alpha))
\right)+o(1),\quad\mbox{ as }n\to\infty.$$
We recall \eqref{eq:diff_free_energies} and \eqref{eq:tmp16}. Taking the expectation with respect to the environment we get
\begin{equation}\label{upperbound}
	f^\bfq(\ee)-f^\bfa(\ee) \leq \lim_{n\to\infty}n^{-d}\E\left[\max_{\alpha\in[\lambda,1]}\sum_{x\in\Lambda_n}\log\left(\alpha(\gamma_x-1)+1\right)\right].
\end{equation}
Let $r,R\in\R$ such that $0<r<1<R<\infty$. The right hand side of \eqref{upperbound} is bounded from above by $I_1(r) + I_2(r,R) + I_3(R)$ where
\[
  I_1(r)=\lim_{n\to\infty}n^{-d}\E\left[\max_{\alpha\in[\lambda,1]}\sum_{x\in\Lambda_n}\log\left(\alpha(\gamma_x-1)+1\right)\ind{0\leq\gamma_x\leq
    r}\right],
\]
and $I_2, I_3$ are defined analogously by exchanging $\ind{0<\gamma_x\leq
    r}$ with $\ind{r<\gamma_x< R}$ and  $\ind{R\leq\gamma_x}$
  respectively. As $r<1$, we have $I_1(r) \leq 0$. Further as $R>1$
the term $I_3(R)$ is maximized at $\alpha=1$, consequently
$$ I_3(R) =\E\log(\gamma_0)\ind{R\leq \gamma_0}.$$
We observe that $\lim_{R\rightarrow \infty }I_3(R) =0$ and further we proceed to $I_2$. Firstly we denote 
$$ f_n(\alpha,r,R):=n^{-d}\sum_{x\in\Lambda_n}
\log(\alpha\gamma_x+1-\alpha)\ind{r<\gamma_x<R},$$
and further let $X_n(\alpha,r,R)=f_n(\alpha,r,R)-\E
  f_n(\alpha,r,R)$. That is
\[
X_n(\alpha,r,R) = n^{-d}\sum_{x\in\Lambda_n}\left[\log(\alpha\gamma_x+1-\alpha)\ind{r<\gamma_x<R}-\E \left(
\log(\alpha\gamma_x+1-\alpha) \ind{r<\gamma_x<R}\right) \right].\]
The summands are centered
independent and bounded. By Hoeffding's inequality we
get
\begin{equation} \mathbb P(|X_n(\alpha,r,R)|>t)\leq 2e^{-2t^2n^d/C_2(\alpha,r,R)^2},\label{eq:hoeffding}\end{equation}
for any $t>0$, where $ C_2(\alpha,r,R)=\log(\alpha R+1-\alpha)-\log(\alpha r+1-\alpha).$
Further we observe that there exists $C_1(\lambda,r,R)>0$ (deterministic) such that 
$$ \max_{\alpha\in[\lambda,1]}\partial_\alpha f_n(\alpha,r,R) \leq 
n^{-d} \sum_{x\in \Lambda_n} \max_{\alpha\in[\lambda,1]}\left(\frac{\gamma_x
-1}{\alpha \gamma_x+1 -\alpha} \ind{r<\gamma_x<R}\right)\leq 
C_1(\lambda,r,R).$$
This let us work with a finite number of values of $\alpha$. Let
$N \in \mathbb{N}$ and $\alpha_i:=(1-\frac iN)\lambda+\frac
iN$. We have
$$ \max_{\alpha\in[\lambda,1]}f_n(\alpha,r,R)\leq
\max_{i=0,\ldots,N}f_n\left(\alpha_i,r,R\right) + C_1(\lambda,r,R)/N.$$
Let $ \bar C_2(r,R)=\max_{\alpha\in[\lambda,1]}
C_2(\alpha,r,R)^2< \infty$. We use \eqref{eq:hoeffding} and the union
bound to get
\begin{align}\label{expectation_Xn}
\mathbb P(\max_{i=0,\ldots,N} |X_n(\alpha_i,r,R)|>t)\leq
2(N+1)e^{-2t^2n^d/\bar C_2(r,R)}.
\end{align}
Now we may write
\begin{multline*}
   I_2(r,R)= \lim_{n\to\infty}\E\max_{\alpha\in[\lambda,1]}f_n(\alpha,r,R)\leq\\
\leq \lim_{n\to\infty}\E(\max_{i=1,\ldots,N}
X_n(\alpha_i,r,R))+\lim_{n\to\infty}\max_{\alpha \in [\lambda,1]}\E
f_n(\alpha,r,R)+\frac{C_1(\lambda,r,R)}{N}.
\end{multline*}
The first term vanishes by \eqref{expectation_Xn}. Taking the limit
$N\uparrow\infty$ we have 
$$I_2(r,R)\leq \lim_{n\uparrow\infty}\max_{\alpha \in [\lambda,1]}\E
f_n(\alpha,r,R).$$
Taking now limits $r\downarrow0$, and $R\uparrow\infty$ and coming back to \eqref{upperbound}, we obtain
\begin{equation}
f^\bfq(\ee)-f^\bfa(\ee)\leq \max_{\alpha \in
	    [\lambda,1]} \E \log(\alpha \gamma_0 +1-\alpha) = \E \log(\lambda
	  \gamma_0 +1-\lambda),\label{eq:tmp20}	
\end{equation}
where the second equality will become apparent shortly.  Indeed, we denote $h(\alpha) := \E \log(\alpha \gamma_0
+1-\alpha)$. Obviously, $h'(\alpha) = \E \frac{\gamma_0 -1}{\alpha (
  \gamma_0-1)+1} \leq \frac{\E (\gamma_0-1)}{\alpha \E (\gamma_0
  -1)+1}= 0$, which follows by the Jensen inequality applied to
$x/(\alpha x +1)$ and the facts that $\gamma_0\geq0$ and
$\E\gamma_0=1$. We notice that \eqref{eq:tmp20} is precisely the bound
announced in Theorem \ref{thm:ineq_quenched_annealed} part a). The
strict inequality between the quenched and annealed free energies
($d\geq 3$) follows easily by the strict Jensen inequality and the
fact that $\E \gamma_0 =1$. 
\qed

\paragraph{Case $d=2$.}

The case $d=2$ requires some slight modification as the stochastic
domination of $\nu_n^{\ell(\ee)}$ by Bernoulli product measures holds in
a weaker sense.
Consequently, we cannot use the same  argument as in
(\ref{stoch_dom}). Indeed, by \cite[Theorem 2.4, (2.13)]{BolVel2001},
for any $\ell(\ee)>0$ and any set $B\subset \Lambda_n$
we have 
\begin{equation}\label{2d_stoch_dom}
\nu_n^{\ell(\ee)} (\mathcal A\cap B = \varnothing)\leq
(1-\lambda)^{|B|},
\end{equation}
{with some $\lambda>0$ depending on $a$ and $\ell(\ee)$, which takes
the form}
\begin{equation}\label{lambda2d}
	\lambda := C_1 \ell(\ee)|\log \ell(\ee)|^{-1/2}
\end{equation}
for some $C_1>0$ when $\ell(\ee)$ is small enough.
Now, using \eqref{comparison_nu_tildenu} and \eqref{2d_stoch_dom} we have
\equnn{  
\mu_n^{\ell(\ee)} (\vert\mathcal A\vert=k)&\leq&
\mu_n^{\ell(\ee)} (\vert\mathcal A\vert\leq k)\leq
\nu_n^{\ell(\ee)} (\vert\mathcal A\vert\leq k)
\\
&=&
\nu_n^{\ell(\ee)} (\exists B\subset\Lambda_n : \vert B\vert=n^d-k\mbox{
  and }
\mathcal A\cap B=\varnothing)\\
&\leq&
\sum_{B\subset\Lambda_n : |B|=n^d-k}\nu_n^{\ell(\ee)} (\mathcal A\cap
B=\varnothing)\leq
\binom{n^d}{k}(1-\lambda)^{n^d-k}.
}
By a direct calculations one checks that for $\tilde{\lambda} := \frac{\lambda}{|\log \lambda|}$ we have
$$ \binom{n^d}{k}(1-\lambda)^{n^d-k}\leq \B_n^{\tilde\lambda}(|\calA|=k),\quad \mbox{
   for } k\leq \lfloor \tilde\lambda n^d\rfloor.$$
Further one may process as for $d\geq3$ by using $\tilde\lambda$
instead of $\lambda$. Consequently we deduce
$$
f^\bfq(\ee)-f^\bfa(\ee)\leq\E\log\left(\frac{\lambda}{|\log\lambda|}\gamma_0+1-\frac{\lambda}{|\log\lambda|}\right).$$
This finishes the proof for $d=2$.
\qed

\paragraph{Acknowledgements.} We would like to warmly acknowledge Yvan
Velenik for introducing us to the topic, and for 
many useful discussions. L.C.\ thanks Fabio Toninelli for his advice.\\
L.C.\ was partially supported by the Swiss National
Foundation. P.M.\ was supported by a Sciex Fellowship grant no.\
10.044.\\
We also thank the reviewer of the first version of this paper for many useful comments.

\bibliographystyle{abbrv}
\bibliography{biblio_short}

\end{document}